\documentclass[11pt, reqno]{amsart}

\usepackage{ifpdf}
\ifpdf
  \usepackage[hyperindex,pagebackref]{hyperref}%
\else
  \expandafter\ifx\csname dvipdfm\endcsname\relax
    \usepackage[hypertex,hyperindex,pagebackref]{hyperref}
  \else
    \usepackage[dvipdfm,hyperindex,pagebackref]{hyperref}
  \fi
\fi

\usepackage{cases}
\usepackage[square, comma, sort&compress, numbers]{natbib}
\usepackage{float}
\usepackage{pifont}
\usepackage{bbding}
\usepackage{amssymb}
\usepackage{mathrsfs}
\usepackage{subfigure}
\usepackage{caption}
\usepackage{geometry}
\usepackage{amsmath, amsthm, amscd, amsfonts, amssymb, graphicx, color}
\usepackage{lineno}

\usepackage{ifpdf}

\newtheorem{theorem}{Theorem}
\newtheorem{thm}{Theorem}
\newtheorem{lemma}{Lemma}
\newtheorem{prop}{Proposition}
\newtheorem{prob}{Problem}

\newtheorem{con}{Conjecture}

\newtheorem{rem}{Remark}

\numberwithin{equation}{section}

\newcommand{\abs}[1]{\left\vert#1\right\vert}

\newcommand{\C}{\mbox{$\mathbb{C}$}}

\newcommand{\D}{\mbox{$\mathbb{D}$}}

\makeatletter
\@namedef{subjclassname@2020}{\textup{2020} Mathematics Subject Classification}
\makeatother

\date{\today}
\begin{document}
\setcounter{page}{1}

\title[Harmonic $K$-quasiconformal Koebe functions]
{Harmonic $K$-quasiconformal Koebe functions: construction and application to Pavlovi\'c's problem
}

\author[Zhi-Gang Wang, Xiao-Yuan Wang, Antti Rasila, and Jia-Le Qiu
]{Zhi-Gang Wang, Xiao-Yuan Wang, Antti Rasila, and Jia-Le Qiu
}

\vskip.20in
\address{\noindent Zhi-Gang Wang\vskip.05in
School of Mathematics and Statistics, Hunan
	First Normal University, Changsha 410205, Hunan, P. R. China.}

 
\vskip.05in
\email{\textcolor[rgb]{0.00,0.00,0.84}{wangmath$@$163.com}}

\vskip.05in
\address{\noindent Xiao-Yuan Wang \vskip.05in
	School of Mathematical Sciences, Liaocheng University, Liaocheng
	252059, Shandong, P. R. 
	China.}

\vskip.05in
\email{\textcolor[rgb]{0.00,0.00,0.84}{mewangxiaoyuan$@$163.com}}
 
\vskip.05in
\address{\noindent Antti Rasila \vskip.05in
	Department of Mathematics with Computer Science, Guangdong Technion-Israel Institute
of Technology, 241 Daxue Road, Shantou 515063, Guangdong, P. R.
	China.}

\address{\noindent
Department of Mathematics, Technion-Israel Institute of Technology, Haifa 3200003, Israel.}
\vskip.05in
\email{\textcolor[rgb]{0.00,0.00,0.84}{antti.rasila$@$iki.fi}; \textcolor[rgb]{0.00,0.00,0.84}{antti.rasila$@$gtiit.edu.cn}}

\vskip.05in
	\address{\noindent Jia-Le Qiu \vskip.03in
		School of Mathematics and Statistics, Changsha University of Science and Technology,
		Changsha 410114, Hunan, P. R. China.}
  \vskip.05in
	\email{\textcolor[rgb]{0.00,0.00,0.84}{qiujiale2023$@$163.com}}





\begin{abstract} 

We first construct the harmonic $K$-quasiconformal Koebe functions, filling a long-standing foundational gap in geometric function theory. This construction provides a unified parametric candidate extremal function framework for conformal mappings, quasiconformal mappings, and harmonic mappings, and we formulate related conjectures for the extremal theory of harmonic $K$-quasiconformal mappings. By combining this construction with Astala and Koskela's $H^p$-theory for quasiconformal mappings, we establish a sharp result concerning the optimal order of harmonic $K$-quasiconformal mappings with bounded Schwarzian norm in harmonic Hardy spaces. Motivated by the work of Chuaqui, Hernández, and Mart\'in’s [Math. Ann. 367, 1099–1122 (2017)], this result gives a partial solution to Pavlovi\'c's  2014 open problem on the embeddings of harmonic quasiconformal mappings into Hardy spaces, and outlines a path toward its complete solution.
\end{abstract} \maketitle


\section{Introduction and statement of the main result}\label{s1} 

Following the seminal work of Clunie and Sheil-Small \cite{cs}, univalent harmonic mappings have emerged as a central topic at the interface of complex analysis, differential geometry, geometric topology, Teichm\"{u}ller theory, nonlinear dynamics, and numerical analysis. Despite remarkable progress over recent decades, several fundamental problems remain strikingly open. Most notably, sharp analogues of Bieberbach's conjecture (de Branges's theorem \cite{de}) and the Koebe one-quarter covering theorem, cornerstones of classical conformal mapping theory, still lack definitive solutions in the harmonic quasiconformal setting. Within geometric function theory, a long-standing bottleneck persists: there exists no effective framework for constructing and rigorously establishing Koebe-type conjectures for harmonic $K$-quasiconformal (\textbf{HQC}) mappings. This critical gap has severely hampered the systematic development of extremal theory and canonical example construction in the field. 

Linear properties, the maximum modulus principle, qualitative results in quasiconformal geometry and the like extend trivially to the setting of harmonic quasiconformal geometric function theory; however, the essence of classical geometric function theory, centered on extremal functions, coefficient conjectures, sharp distortion inequalities and Loewner theory, cannot be extended trivially at all, and a completely new theoretical framework must be established. This is precisely the core value and research significance of harmonic quasiconformal geometric function theory.



Let $\mathcal{H}$ denote the class of complex-valued harmonic functions
$f=h+\overline{g}$ in the open unit disk $\D$, normalized by the conditions $$f(0)=f_{z}(0)-1=0,$$ which have the
form
\begin{equation}\label{111}
f(z)=z+\sum_{n=2}^{\infty}a_nz^n+\overline{\sum_{n=1}^{\infty}b_nz^n}.
\end{equation}

Denote by $\mathcal{A}$ the class of normalized analytic functions, which is a subclass of
$\mathcal{H}$ with $g(z)\equiv 0$.
Let $\mathcal{S_{\mathcal{H}}}$ be the class of harmonic functions $f\in\mathcal{H}$
that are univalent and sense-preserving in $\D$. Moreover, we denote by $\mathcal{S}^{0}_{\mathcal{H}}$ the subclass of $\mathcal{S_{\mathcal{H}}}$ with the additional condition $f_{\overline{z}}(0)=0$. Since there exist reciprocal transformations between the classes $\mathcal{S_{\mathcal{H}}}$
and $\mathcal{S}^{0}_{\mathcal{H}}$, we usually focus our attention on the class $\mathcal{S}^{0}_{\mathcal{H}}$.  Clearly, the classical class $\mathcal{S}$ of normalized univalent analytic functions is a subclass of the class
$\mathcal{S}_\mathcal{H}^0$. 

\subsection{Harmonic quasiconformal mappings}
We say that $f$ belongs to the class $\mathcal{S}_\mathcal{H}(K)$ of \textbf{HQC} mappings, where $K\ge 1$ is a constant, if $f\in\mathcal{S}_\mathcal{H}$ and its dilatation defined by $$\omega:=\frac{g'}{h'}$$ satisfies the condition 
$\abs{\omega}\leq k$, where $k\in [0,1)$ is given by
$$
k:=\frac{K-1}{K+1}\quad(K\geq 1).
$$ A function $f$ is called a harmonic quasiconformal mapping, if it belongs to $\mathcal{S}_\mathcal{H}(K)$ for some $K\ge 1$.

For convenience, we define $$\mathcal{S}^0_\mathcal{H}(K):=\mathcal{S}_\mathcal{H}(K)\cap\mathcal{S}^0_\mathcal{H}.$$
Note that $$\mathcal{S}^0_\mathcal{H}(1)=\mathcal{S}_\mathcal{H}(1)=\mathcal{S}.$$

In recent years,
Kalaj \cite{ka} obtained the extended form of Lindel\"{o}f theorem for harmonic quasiconformal mappings.
Kalaj \cite{ka2} established sharp Riesz-type inequalities for harmonic mappings on the open unit disk in the complex plane.
Wang, Shi, and Jiang \cite{wsj} investigated harmonic quasiconformal mappings associated with asymmetric vertical strips. 
Sun, Rasila, and Jiang \cite{srj} considered linear combinations of harmonic quasiconformal mappings convex
in one direction.
Chuaqui, Hern\'{a}ndez, and Mart\'{i}n \cite{chm}
studied the affine and linear invariant families of
harmonic mappings, and obtained the sharp order of harmonic mappings with bounded Schwarzian norm.
Partyka, Sakan, and Zhu \cite{psz}  provided various properties of harmonic quasiconformal mappings with the
convex holomorphic part.
Liu and Zhu \cite{lz} studied the Riesz conjugate functions theorem for harmonic quasiconformal
mappings. 
Todor\v{c}evi\'{c}'s monograph \cite{t} gives recent developments on higher-dimensional harmonic quasiconformal mappings and hyperbolic type metrics.



\subsection{Hardy spaces for harmonic and quasiconformal mappings}
Let $0<p \leq \infty$. A complex-valued function $f$ in $\D$ is said to be in the {\it Hardy space} for $p$, if the integral mean
 \[M_p(r,f):= \begin{cases}
\left(\frac{1}{2\pi}\int_0^{2\pi}\abs{f(re^{i\theta})}^pd\theta\right)^{\frac{1}{p}}&(0<p<\infty),\\ \ \
\sup\limits_{{\abs{z}=r}}\abs{f(z)}&(p=\infty),
\end{cases}\]
is defined for all $r\in(0,1)$, and
$$\sup\limits_{0<r<1}M_p(r, f)<\infty.$$ 

We record the following special cases of this definition. The classical Hardy space of analytic functions is denoted by $H^p$ 
and the one of $K$-quasiconformal mappings \cite{ak} is denoted by $H^p_K$. Similarly, we say that a harmonic function $f$ is in a \textit{harmonic Hardy space} $h^p$, if it is harmonic and satisfies the above condition. We denote the space of \textbf{HQC} mappings by $h^p_K$. 

\subsection{The order of a Hardy space} 

For a class $\mathcal{G}$ of complex-valued functions in $\D$, we define its {\it order} as a constant $\gamma\geq 0$ such that $\mathcal{G}$ is contained the corresponding Hardy space $H^p$ for all $p\in(0,\gamma)$.

We recall the following special case of the result due to Astala and Koskela \cite[Theorem 3.2]{ak}.

\vskip.10in

\begin{thm} \label{thmB}
 If $f$ is a $K$-quasiconformal mapping in $\D$, then $f\in H^p_K$ for $$0<p<\frac{1}{2K}\quad(K\geq1),$$ and the constant $1/(2K)$ is sharp.
 \end{thm}

\vskip.10in

In 1990, Abu-Muhanna and Lyzzaik \cite[Theorem 2]{al} established the following result showing that for a harmonic mapping $f=h+\overline{g}$, the functions $h$ and $g$ belong to the Hardy space $H^p$, and $f$ belongs to harmonic Hardy spaces $h^p$, respectively, for a suitable choice of $p$.

\vskip.10in

\begin{thm} \label{thmC}
 If $f=h+\overline{g}\in \mathcal{S}_\mathcal{H}$, then $h,\, g\in H^p$ and $f\in h^p$ for every $p$ with $$0<p<\frac{1}{(2\beta+2)^{2}},$$
where $$\beta:=\sup\limits_{f\in\mathcal{S}_\mathcal{H}}\abs{a_2}.$$
\end{thm}

In 1996, the order $$\frac{1}{(2\beta+2)^{2}}$$ in Theorem \ref{thmC} was improved to $1/\beta^{2}$ by Nowak in \cite[Theorem 1]{n}, where she conjectured that the order is actually $1/\beta$. Moreover, she obtained the sharp results $$f\in h^p\quad\left(0<p<\frac{1}{2}\right)$$ for the class $\mathcal{K}_\mathcal{H}$ of convex harmonic mappings, and 
$$f\in h^p\quad\left(0<p<\frac{1}{3}\right)$$ for the class $\mathcal{C}_\mathcal{H}$ of close-to-convex harmonic mappings. Even today, proving  Nowak's conjecture is seen as a 
\textit{challenging} problem.

We also note that Aleman and Mart\'{\i}n \cite{am} proved that convex harmonic mappings are not necessarily in $h^{1/2}$, which provided a negative answer
to a question raised by Duren \cite[p. 153, Sect. 8.5]{d}. Laitila, Nieminen, Saksman, and Tylli \cite{ln} considered the rigidity of composition operators involving harmonic Hardy spaces $h^p$. Melentijevi\'{c} and Bo\v{z}in \cite{mb} obtained sharp Riesz-Fej\'{e}r inequality for harmonic Hardy spaces $h^p$.

An analytic function $h$ in $\D$ has valency $n$ if $h$ takes no value
more than $n$ times. More generally, for a function $h$ analytic in $\D$, let $W(R)$ denote
the area (regions covered multiply being counted multiply) of the image of $\D$ under
$h$ that lies in the closed disk $\abs{w}\leq R$. If $W(R)\leq n\pi R^2$ for all $R>0$, where $n$ is a positive number, then we say that $h$ has mean valency $n$ (cf. \cite{ds}).

In 2024, Das and Sairam Kaliraj \cite[Theorem 1]{ds}
proved that Nowak's conjecture is true with the additional condition that $h^\prime$ has a \textit{finite mean valency}. 
Noting that the conjectured $\beta$ of the class $\mathcal{S}_\mathcal{H}$ is equal to $3$ (see \cite{cs}), and it is consequently conjectured by Duren \cite[p. 152, Sect. 8.5]{d} that $$\mathcal{S}_\mathcal{H}\subset h^{{1}/{3}}.$$

\subsection{Affine and linear invariant families}
For $f=h+\overline{g}\in\mathcal{S}_\mathcal{H}$, following the terminology of \cite{s}, a family is said to be an \textit{affine
and linear invariant family} ({AL} family) if it is closed under both the \textit{Koebe transform}
$$
K_{\zeta}(f)(z):=\frac{f\left((z+\zeta)/\left(1+\overline{\zeta} z\right)\right)-f(\zeta)}{\left(1-\abs{\zeta}^2\right)h'(\zeta)}
\quad(\abs{\zeta}<1)
$$
and \textit{affine changes}
$$A_{\varepsilon}(f)(z):=\frac{f(z)-\overline{\varepsilon f(z)}}{1-\overline{\varepsilon}g'(0)}
\quad(\abs{\varepsilon}<1).$$

In 1990, Sheil-Small \cite{s} 
offered an in-depth study of affine and linear invariant families $\mathcal{F}$ of
harmonic mappings $f=h+\overline{g}$ in $\D$. The {\it order} of an AL family is given by
$$\alpha(\mathcal{F}):=\sup\limits_
{f\in\mathcal{F}}
\abs{a_2(f)}=\frac{1}{2}
\sup\limits_
{f\in\mathcal{F}}
\abs{h''(0)}.$$
Indeed, many important characterizations of the AL family $\mathcal{F}$ are determined
by its {\it order}. Clearly, the class $\mathcal{S_{\mathcal{H}}}$ is a typical example of the affine and linear invariant families.

\subsection{Harmonic pre-Schwarzian and Schwarzian derivatives}
The classical pre-Schwarzian and Schwarzian derivatives of locally univalent analytic functions $h$ (cf. \cite{du}),
which are, respectively, defined by
$$P_h(z):=\frac{h''(z)}{h'(z)}$$
and
$$S_h(z):=\left(\frac{h''(z)}{h'(z)}\right)'-\frac{1}{2}\left(\frac{h''(z)}{h'(z)}\right)^2.$$


For a locally univalent sense-preserving harmonic mapping $f=h+\overline{g}$ in $\D$, Hern\'{a}ndez and Mart\'{i}n \cite{hm} introduced
the pre-Schwarzian and Schwarzian derivatives of $f$ as follows:
$$P_f:=\left(\log J_f\right)_z=P_h-\frac{\omega'\,\overline{\omega}}{1-\abs{\omega}^2}$$ and
\begin{align*}\begin{split}S_f&:=\left(\log J_f\right)_{zz}-\frac{1}{2}\left[\left(\log J_f\right)_z\right]^2\\&\, \, =S_h+\frac{\overline{\omega}}{1-\abs{\omega}^2}\left(\frac{h''}{h'}\omega'-\omega''\right)
-\frac{3}{2}\left(\frac{\omega'\,\overline{\omega}}{1-\abs{\omega}^2}\right)^2,\end{split}\end{align*}
where $$J_f:=\abs{f_z}^2-\abs{f_{\overline{z}}}^2$$ is the Jacobian of $f$.
The corresponding pre-Schwarzian norm $\|P_f\|$ and Schwarzian norm
$\|S_f\|$ of $f$ are defined, respectively, by
$$\|P_f\|:=\sup\limits_
{z\in\mathbb{D}}
\abs{P_f(z)}\left(1-\abs{z}^2\right)$$
and
$$\|S_f\|:=\sup\limits_
{z\in\mathbb{D}}
\abs{S_f(z)}\left(1-\abs{z}^2\right)^2.$$

We note that Chuaqui, Duren, and Osgood \cite{cdo} have given an alternative definition of harmonic Schwarzian derivative
with restriction on its dilatation.  

\subsection{A problem of Pavlovi\'{c}}
Because univalent harmonic mappings are not necessarily \textit{quasiconformal}, the constant $1/(2K)$ in Theorem \ref{thmB} may not be sharp for the class of harmonic $K$-quasiconformal mappings. Based on Theorems \ref{thmB} and \ref{thmC}, Pavlovi\'{c} \cite[p. 315, Problem 10.2]{pav} in 2014 proposed the following natural problem.

\begin{prob}\label{p1}
{Find the sharp $$p_0=p_0(K)$$ such that every harmonic $K$-quasiconformal mapping belongs to $h^p_K$ for $p<p_0$.}
\end{prob}

\begin{rem}\label{r02}
{\rm One of keys to solve the above problem 
is determination of $\sup\limits_{f\in\mathcal{S}_\mathcal{H}(K)}\abs{a_2}$ of the family $\mathcal{S}_\mathcal{H}(K)$.
But, at present, it seems that this is a \textit{hard} problem (see \cite[Proposition 1.6]{cp} and 
Conjecture \ref{conj1} in Sect. \ref{s3}
).}
\end{rem}

Curiously, Das, Huang, and Rasila \cite[Theorem 3]{dhr} established the optimal orders for convex and close-to-convex harmonic $K$-quasiconformal mappings to belong to harmonic Hardy spaces, respectively; notably, these results are independent of the quasiconformal constant $K$, i.e.,
$$f\in h^p\quad\left(0<p<1\right)$$ for the class $\mathcal{K}_\mathcal{H}(K)$ of 
convex harmonic quasiconformal mappings, and 
$$f\in h^p\quad\left(0<p<\frac{1}{2}\right)$$ for the class $\mathcal{C}_\mathcal{H}(K)$ of 
close-to-convex harmonic quasiconformal mappings. 

Note that Ponnusamy, Qiao, and Wang \cite[Theorem 4.2]{pqw} also obtained a sharp result regarding uniformly locally univalent harmonic quasiconformal mappings belong to harmonic Hardy space.

 Denote by  $\widehat{\mathcal{S}_\mathcal{H}}(K)$ the subclass of $\mathcal{S}_\mathcal{H}(K)$, 
where $h'$ has a \textit{finite mean valency}. We now give the following characteristic of the class $\widehat{\mathcal{S}_\mathcal{H}}(K)$, which will play a crucial role in our study.

\begin{prop}\label{prop1}
\textit{Suppose that $f=h+\overline{g}\in\widehat{\mathcal{S}_\mathcal{H}}(K)$. If $$\phi(K):=\sup\limits_{f\in\widehat{\mathcal{S}_\mathcal{H}}(K)}\abs{a_2}.$$
Then \begin{enumerate}
\item for $\phi(K)\leq 2K$, we have $$f\in h^p\quad \left(0<p<\frac{1}{2K}\right);$$
\vskip.05in
\item for $\phi(K)>2K$, we have $$f\in h^p\quad \left(0<p<\frac{1}{\phi(K)}\right).$$
\end{enumerate}
Both of these bounds are sharp.}
\end{prop}

\begin{proof}
By the fact that $\phi(K)$ is a real function with respect to $K$,
$\phi(K)\geq1$ for $K\geq1$ (see \cite{p}), using the similar method as in  Das and Sairam Kaliraj \cite[Theorem 1]{ds}, 
we confirm that if $f\in\widehat{\mathcal{S}_\mathcal{H}}(K)$, 
then $f\in h^p$ for $$0<p<\frac{1}{\phi(K)}.$$
Note that, by Theorem \ref{thmB} and taking the value of $$\min\limits_{K\geq 1}\left\{\frac{1}{\phi(K)},\ \frac{1}{2K}\right\}.$$ We obtain the assertion of
Proposition \ref{prop1}.
\end{proof}

For convenience, we denote the subclass of ${\mathcal{S}_\mathcal{H}}(K)$ (resp. $\widehat{\mathcal{S}_\mathcal{H}}(K)$) with
\begin{equation*}
 \|S_f\|\leq \lambda
\end{equation*}
by ${\mathcal{S}_\mathcal{H}}(K,\lambda)$ (resp. $\widehat{\mathcal{S}_\mathcal{H}}(K,\lambda)$), where $K\geq1$ and $\lambda\geq0$. 
Now, we establish the following relationship between
the family $\widehat{\mathcal{S}_\mathcal{H}}(K,\lambda)$ of \textbf{HQC} mappings with bounded Schwarzian norm and harmonic Hardy spaces $h^p$.

\begin{theorem}\label{t1}
Let $f\in\widehat{\mathcal{S}_\mathcal{H}}(K,\lambda)$ with $K\geq1$ and $\lambda\geq0$. 
Then
\begin{enumerate}
\item for $0\leq\lambda\leq6$ and $K\geq1$, we have $$f\in h^p\quad \left(0<p<\frac{1}{2K}\right);$$
\vskip.05in
\item for $\lambda>6$ and $K\geq K_1>1$, we have $$f\in h^p\quad \left(0<p<\frac{1}{2K}\right);$$
\vskip.05in
\item for $\lambda>6$ and $1\leq K<K_1$, we have $$f\in h^p\quad \left(0<p<\frac{1}{\varphi(K,\lambda)}\right),$$
\end{enumerate}
where
\begin{equation}\label{1}\varphi(K,\lambda):=\sup\limits_{f\in\widehat{\mathcal{S}_\mathcal{H}}(K,\lambda)}\abs{a_2}
=\sqrt{1+\frac{\lambda}{2}+\frac{1}{2}\left(\frac{K-1}{K+1}\right)^2}+\frac{K-1}{2\left({K+1}\right)},\end{equation}
and 
$K_1$ is the unique solution in $(1,+\infty)$ of the equation
\begin{equation}\label{2}16K^4+24K^3-(2\lambda-11)K^2-2(2\lambda-1)K-2\lambda-5=0.\end{equation}
All of these results are sharp.
\end{theorem}

\begin{rem}
{\rm For $K=1$ and $0\leq\lambda\leq6$, 
 the first part of Theorem \ref{t1} matches the classical analytic case, as originally examined by Pommerenke \cite{p} in 1964.
}
\end{rem}


This paper employs the weighted shearing technique to construct the first harmonic $K$-quasiconformal Koebe functions for the class, resolving a long-standing core open problem in the field. The constructed harmonic quasiconformal Koebe-type function serves as a candidate extremal function for this class, and a rigorous proof for the corresponding coefficient conjecture remains elusive to date, its difficulty being analogous to that of the classical Bieberbach conjecture. This construction enables the systematic extension of classical results from conformal mapping theory to the harmonic quasiconformal setting, thereby inaugurating a new systematic research program in geometric function theory.
We also establish the sharp order for harmonic $K$-quasiconformal mappings with bounded Schwarzian norm. Furthermore, employing a fundamental result on the $H^p$ theory for quasiconformal mappings due to Astala and Koskela \cite{ak}, we determine the optimal order for the family of harmonic $K$-quasiconformal mappings with bounded Schwarzian norm that lie in harmonic Hardy spaces. This partially resolves an open problem posed by Pavlovi'{c} \cite{pav} in 2014.

\vskip.20in

\section{Preliminaries}

In this section, we provide several required preliminary results.

\subsection{Estimation of $\abs{a_2}$ of a univalent harmonic mapping}
The problem of determining the sharp upper bound of $|{a_2}|$ in the class
$\mathcal{S}^0_\mathcal{H}$
is still an \textit{open} problem. Clunie and Sheil-Small \cite{cs} once proved that 
$$\abs{a_2}<12172.$$
Later, Sheil-Small \cite{s} improved it to $$\abs{a_2}<57.$$ The estimate $$\abs{a_2}<49$$ was
proved by Duren \cite[p. 96, Sect. 6.3]{d}. 
To our knowledge, the latest bound is $$\abs{a_2}<20.9197,$$ which was established
by Abu-Muhanna, Ali, and Ponnusamy \cite{aap} in 2019. 
It has been conjectured by Clunie and Sheil-Small \cite{cs}  that $$\abs{a_2}\leq\frac{5}{2},$$
where the equality is reached by the \textit{harmonic Koebe function} $\mathbb{K}$ defined by
$$\mathbb{K}(z):=\frac{z-\frac{1}{2}z^2+\frac{1}{6}z^3}{(1-z)^3}+\overline{\frac{\frac{1}{2}z^2+\frac{1}{6}z^3}{(1-z)^3}}.$$
This is the expected candidate for extremal mappings in the class $\mathcal{S}^0_\mathcal{H}$ of univalent harmonic mappings.

\subsection{Estimates of $\abs{a_2}$ of harmonic $K$-quasiconformal mappings}
We note that the problem of finding the sharp upper bound of $\abs{a_2}$ in the class $\mathcal{S}^0_\mathcal{H}(K)$
is also {\it challenging}. In 2022, Chen and Ponnusamy \cite[Proposition 1.6]{cp} proved the following rough estimate on $|a_2|$ for $f\in\mathcal{S}^0_\mathcal{H}(K)$.
\vskip.10in

\begin{thm} \label{thmA}
 For $K\geq1$, let $f=h+\overline{g}\in {\mathcal{S}_\mathcal{H}^0(K)}$, where $$h(z)=z +\sum_{n=2}^{\infty}a_n z^n\ {\textit and} \ g(z)=\sum_{n=2}^{\infty}b_n z^n.$$ Then 
$$|a_2|\leq\frac{32}{(1+1/K)^4}+\frac{64}{(1+1/K)^3}-2.$$
\end{thm}

\begin{rem}\label{r01}
{\rm For $K=1$ in Theorem \ref{thmA}, we observe that the result reduces to the case $\mathcal{S}^0_\mathcal{H}(1)=:\mathcal{S}$
of conformal mappings. 
At this time, we find that $$|a_2|\leq 8.$$ However, by Bieberbach's theorem, the sharp bound $$|a_2|\leq 2\ {\rm or}\ \alpha(\mathcal{S})=2$$ holds for $f\in\mathcal{S}$,  which shows that there still is substantial room for improvement on the upper bound of $\abs{a_2}$ for the family $\mathcal{S}^0_\mathcal{H}(K)$.}
\end{rem}









In 2017, Chuaqui, Hern\'{a}ndez, and Mart\'{i}n \cite[Theorem 1]{chm} proved the following result regarding the order
of the affine and linear invariant family $\mathcal{F}_{\lambda}$. 
We observe that this result will play an important role in our study.

\vskip.10in
\begin{thm} \label{thmD}
The order of the family $\mathcal{F}_{\lambda}$ is given by 
\begin{align*}
\begin{split}\alpha\left(\mathcal{F}_{\lambda}\right)=\sqrt{1+\frac{\lambda}{2}+\frac{1}{2}\sup\limits_{f\in\mathcal{F}_{\lambda}^0}
\abs{g''(0)}^2}+\frac{1}{2}\sup\limits_{f\in\mathcal{F}_{\lambda}^0}
\abs{g''(0)}. \end{split}
\end{align*}
\end{thm}



\vskip.20in
\section{Harmonic $K$-quasiconformal Koebe functions}\label{s3}

A long-standing bottleneck in the study of harmonic quasiconformal mappings lies in the absence of a canonical Koebe-type extremal function, which has significantly impeded the advancement of its extremal theory as well as the generalization of classical conformal mapping results to the harmonic quasiconformal framework. In this section, we remedy this critical gap by constructing such functions for the first time.

In what follows, we construct the harmonic $K$-quasiconformal Koebe functions
$$f_k(z)=h(z)+\overline{g(z)}$$
via the \textit{weighted shearing technique}, an approach adapted from the classical shearing method introduced by Clunie and Sheil-Small \cite{cs}.
The defining relations for its components are given by
$$h(z)-g(z)=\frac{z}{(1-z)^2}$$
and
$$\omega(z)=\frac{g'(z)}{h'(z)}=k z\quad (0\leq k<1).$$


Solving the above system of differential equations and normalization, for $0\leq k<1$, we obtain \begin{align}\begin{split}\label{31}
  f_k(z)&=\frac{1}{(k-1)^3}{\left[\frac{(k-1)(1-3k+2kz)z }{(1-z)^2}+k(k+1) \log\left(\frac{1-z}{1-k z}\right)\right]}\\ & \qquad\ \ \ +\frac{k}{(k-1)^3}\overline{{\left[\frac{(1-k)(1+k-2z)z}{(1-z)^2}+(k+1) \log\left(\frac{1-z}{1-k z}\right)\right]} }.\end{split}
\end{align}
The power series form of $f_k$ is 
\begin{equation}\label{032}
\begin{aligned}
    f_k(z)=z+\sum_{n=2}^{\infty}A(n,k)z^n+\sum_{n=2}^{\infty}B(n,k)\overline{z}^n,
\end{aligned}
\end{equation}
where \begin{equation}\begin{aligned}\label{033}
A(n,k):=\frac{(1-k)^2n^2-2k(1-k)n+k(1+k)(1-k^n)}{(1-k)^3\,n}
\end{aligned}
\end{equation} and \begin{equation}\begin{aligned}\label{034}
B(n,k):=\frac{k(1-k)^2n^2-2k(1-k)n+k(1+k)(1-k^n)}{(1-k)^3\,n}.
\end{aligned}
\end{equation}

The images of the unit disk under analytic Koebe function $f_{0}(z)$, 
$f_{1/5}(z)$, $f_{2/5}(z)$, $f_{3/5}(z)$, $f_{4/5}(z)$ 
and harmonic Koebe function $\mathbb{K}(z)$ are illustrated, respectively, in Figure \ref{Fig_koebe_table}.

\begin{figure}[H] 
\centering 
\subfigure[Image of analytic Koebe function $f_0(z)$.]{ 
\label{Fig.sub.1} 
\includegraphics[height=6.0cm]{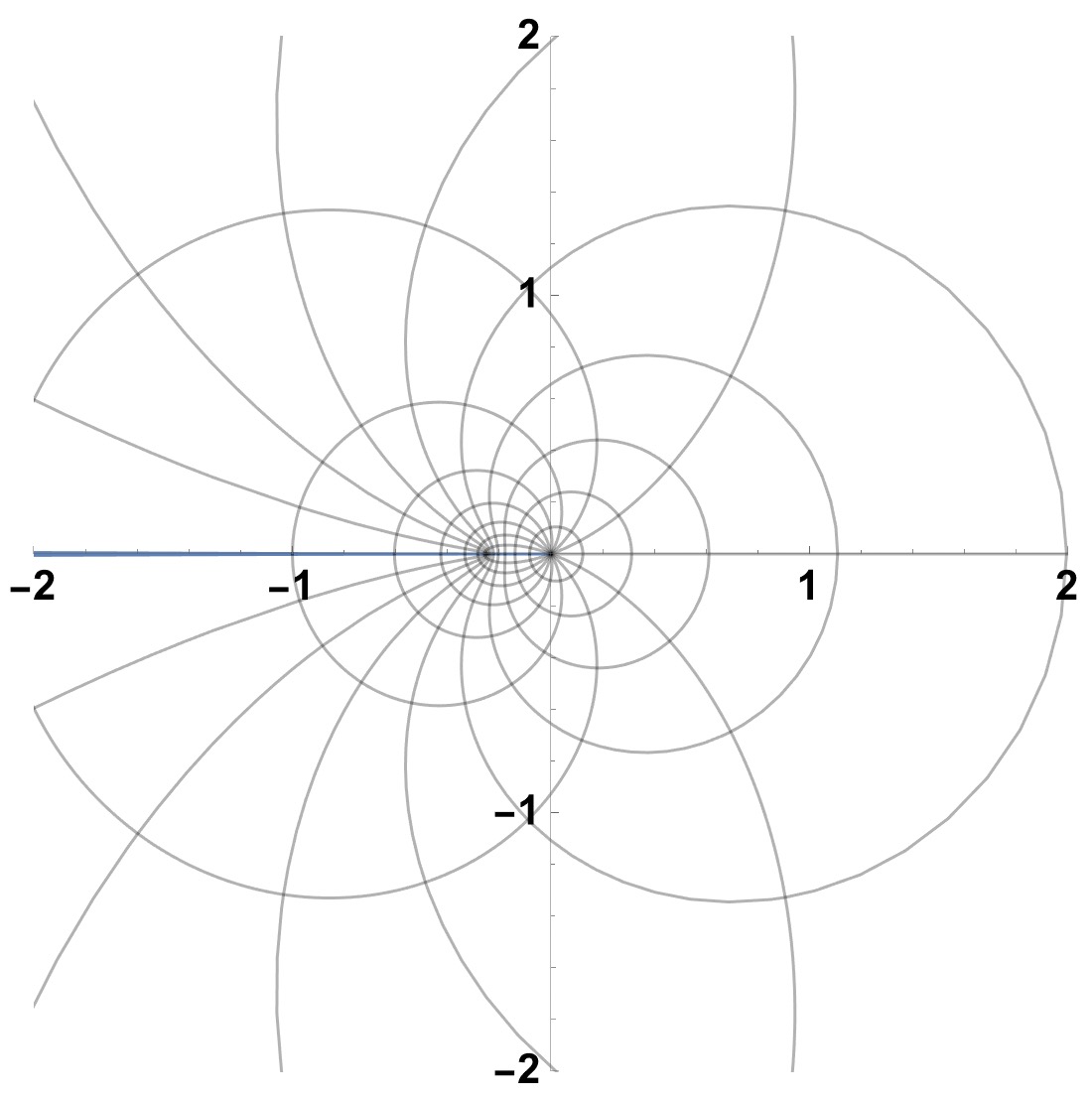}} 
\subfigure[Image of $f_{1/5}(z)$.]{ 
\label{Fig.sub.2} 
\includegraphics[height=6.0cm]{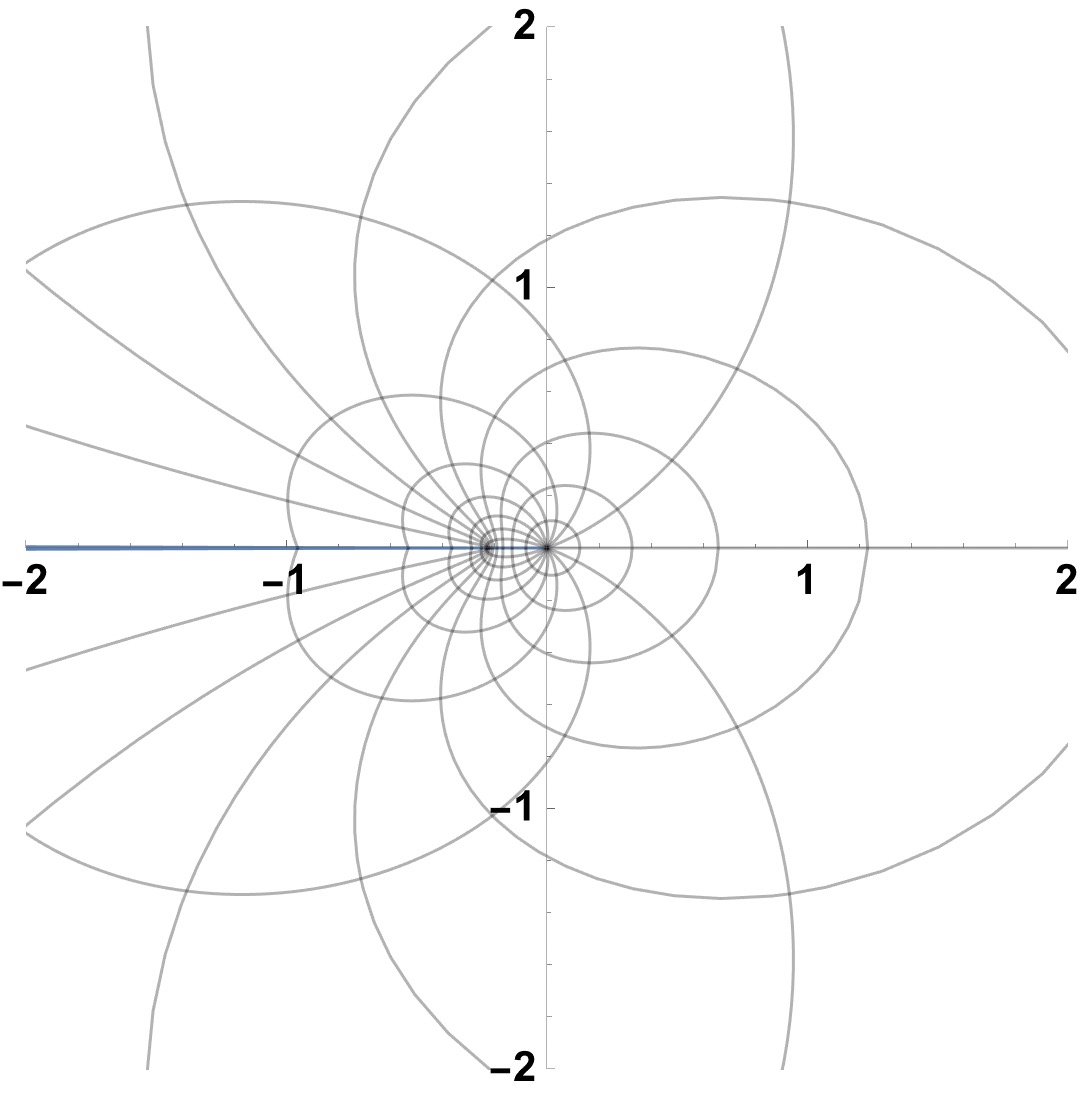}} 
\end{figure}

\begin{figure}[H] 
\centering 
\subfigure[Image of $f_{2/5}(z)$.]{ 
\label{Fig.sub.3} 
\includegraphics[height=6.0cm]{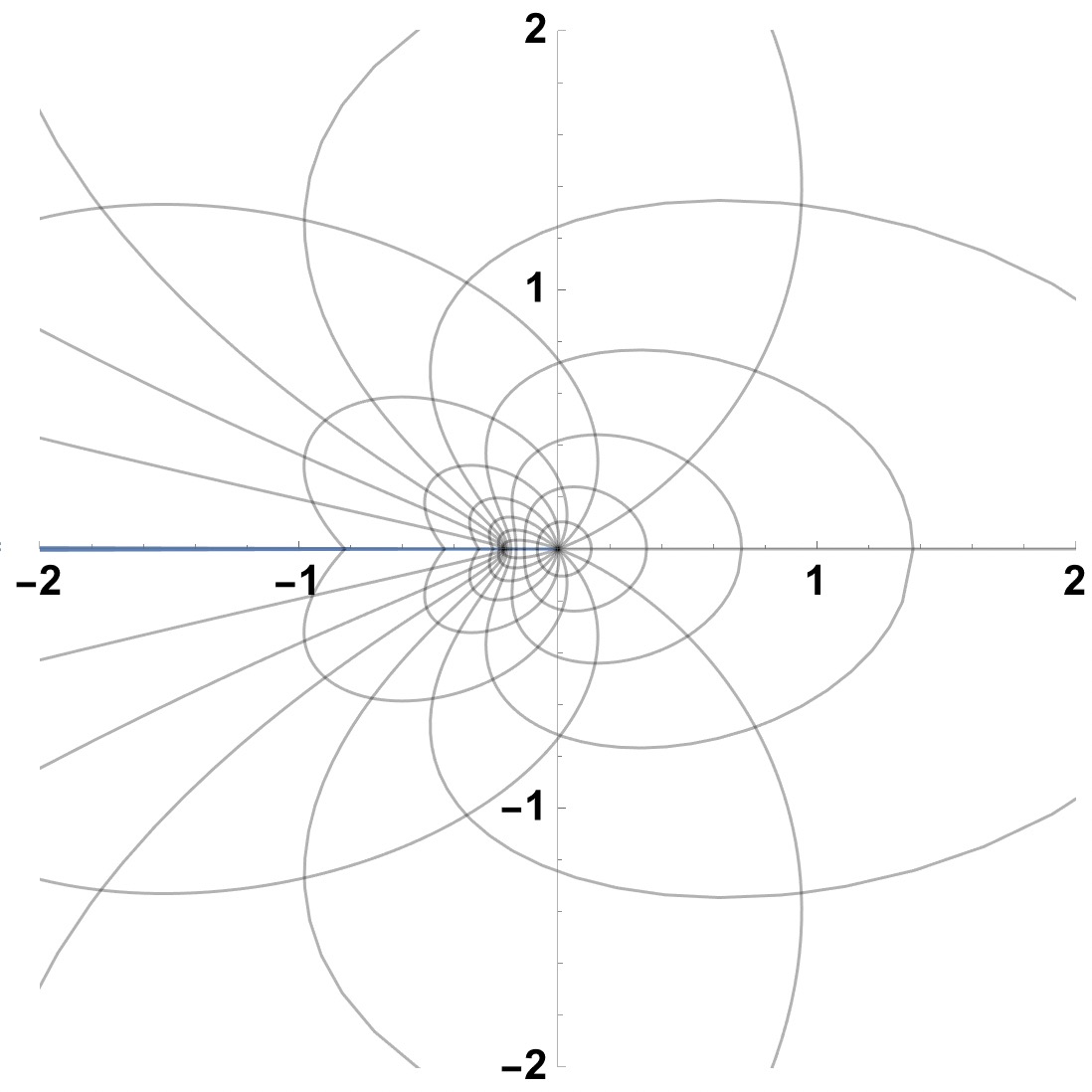}} 
\subfigure[Image of $f_{3/5}(z)$.]{ 
\label{Fig.sub.4} 
\includegraphics[height=6.0cm]{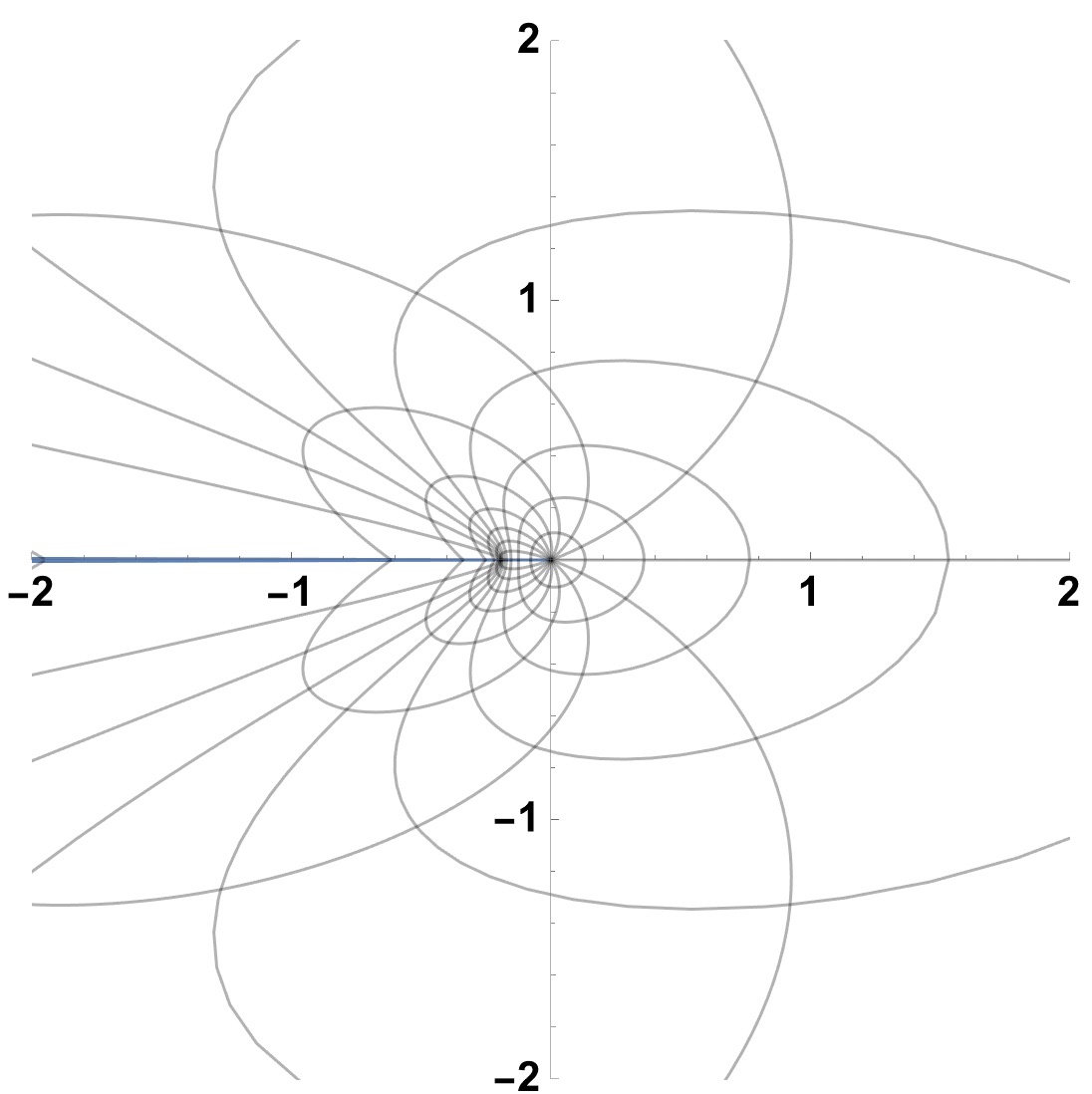}} 
\end{figure}

\begin{figure}[H] 
\centering 
\subfigure[Image of $f_{4/5}(z)$.]{ 
\label{Fig.sub.5} 
\includegraphics[height=6.0cm]{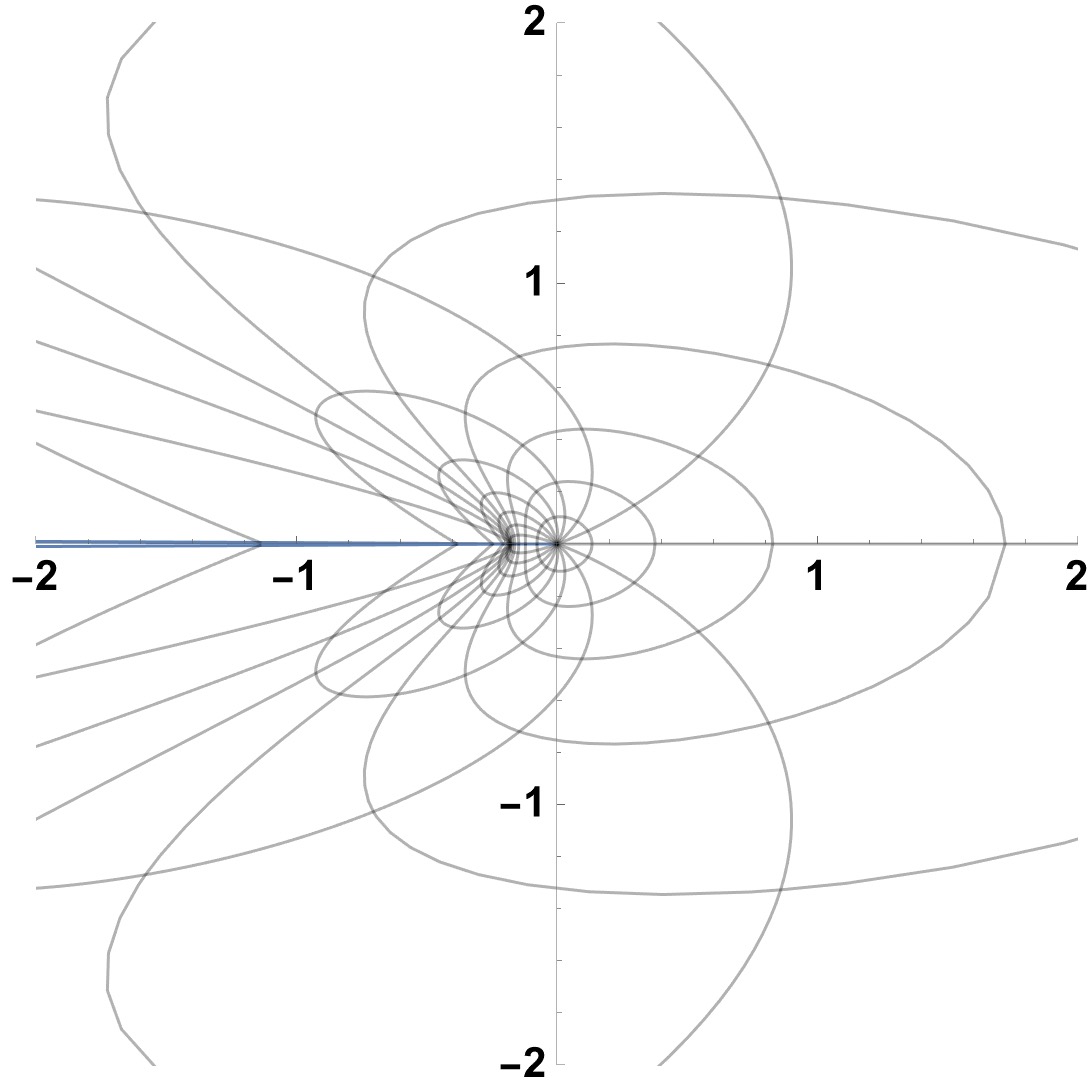}} 
\subfigure[Image of harmonic Koebe function $\mathbb{K}(z)$.]{ 
\label{Fig.sub.6} 
\includegraphics[height=6.0cm]{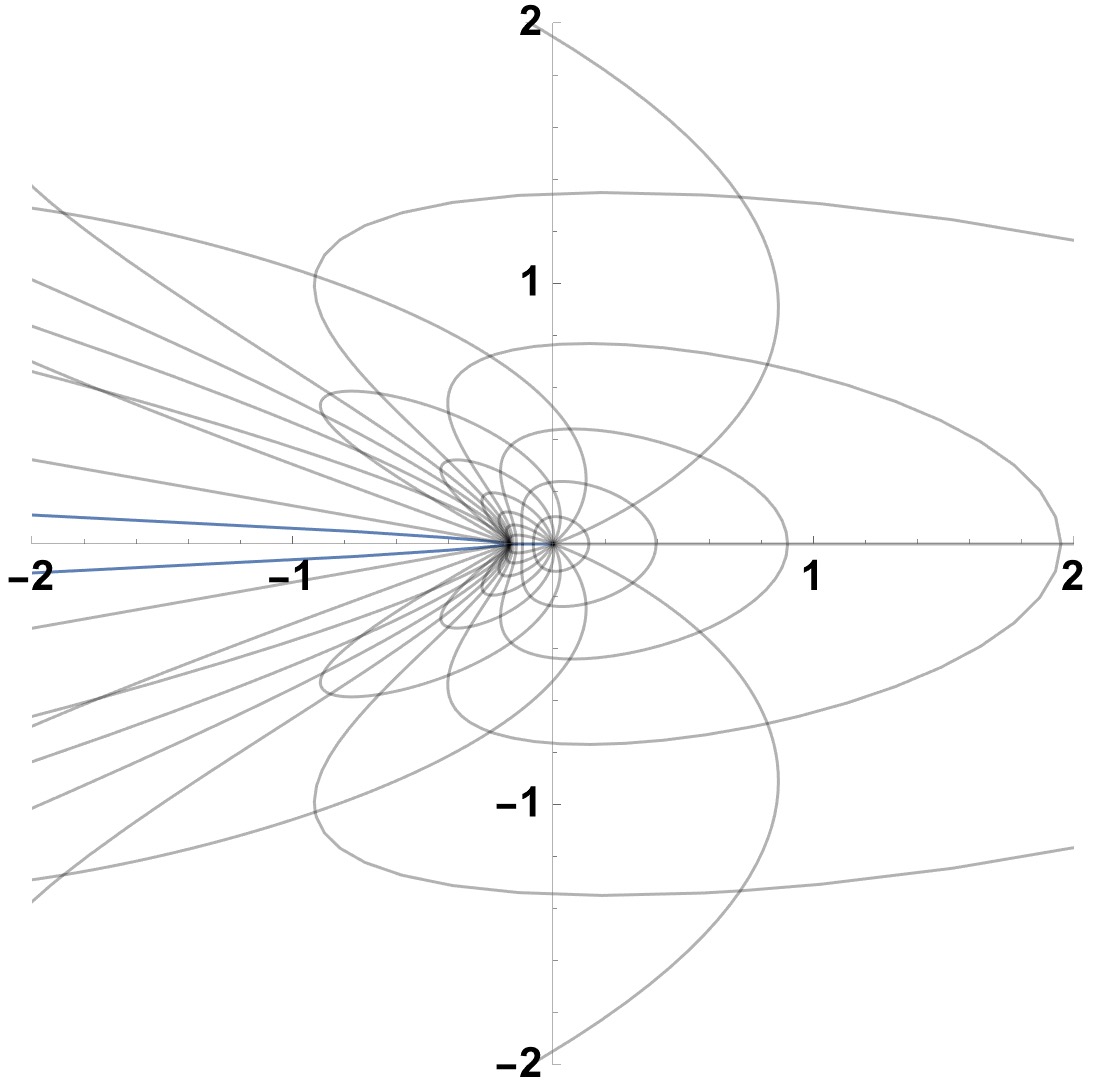}} 
\caption{\small{Images of the unit disk under the mappings $f_{0}(z)$, $f_{1/5}(z)$, $f_{2/5}(z)$, $f_{3/5}(z)$, $f_{4/5}(z)$, 
and $\mathbb{K}(z)$.}}
\label{Fig_koebe_table} 
\end{figure}




By noting that the coefficient conjecture of Clunie and Sheil-Small \cite{cs} holds for the close-to-convex, starlike and typically real harmonic mappings, and harmonic quasiconformal mappings is a subclass of univalent harmonic mappings, the Bieberbach-type coefficient conjecture for \textbf{HQC} is confirmed valid by Li and Ponnusamy \cite{lp} for close-to-convex, starlike, and typically real harmonic quasiconformal mappings (see also \cite{dhr}). For recent investigations on Koebe-type functions in the setting of harmonic quasiconformal mappings, see \cite{wqr}.


It is natural to conjecture that the function $f_k(z)$ serves as the extremal mapping for the class $\mathcal{S}^0_\mathcal{H}(K)$ of \textbf{HQC} mappings. This not only opens a novel perspective on the class $\mathcal{S}^0_\mathcal{H}(K)$, but also provides a fundamental bridge connecting the theories of classical conformal mappings, harmonic mappings, quasiconformal mappings, Teichmüller spaces, and even SLE theory.

Indeed, the above formulas lead us to conjecture that the extremal function in $\mathcal{S}^0_\mathcal{H}(K)$ is unique and coincides with the harmonic $K$-quasiconformal Koebe functions.

\begin{con}\label{conj00}{\rm (\textbf{HQC Bieberbach conjecture})}
 For $K\geq1$, let $f=h+\overline{g}\in {\mathcal{S}_\mathcal{H}^0(K)}$, where $$h(z)=z +\sum_{n=2}^{\infty}a_n z^n\ {\rm and} \ g(z)=\sum_{n=2}^{\infty}b_n z^n.$$ Then 
$$\abs{a_n}\leq A(n,k),\, \abs{b_n}\leq B(n,k),\,{\rm and}\, \abs{\abs{a_n}-\abs{b_n}}\leq n\,\, {\rm for}\,\, n\geq 2,$$
where $A(n,k)$ and $B(n,k)$ are given by \eqref{033} and \eqref{034}, respectively, with $$k=\frac{K-1}{K+1}\quad(K\geq1).$$
\end{con}

However, this somewhat is premature in view that the related conjecture for the class $\mathcal{S}_\mathcal{H}^0$ (see e.g., \cite[pp. 87--88]{d}) remains 
an \textit{open} problem.

By the above coefficient conjecture for the class $\mathcal{S}^0_\mathcal{H}(K)$, we obtain the conjectured sharp upper bound of $\abs{a_2}$, which will play a distinguished role to study various characterizations for the class  $\mathcal{S}^0_\mathcal{H}(K)$ (or even the class $\mathcal{S}^0_\mathcal{H}$). 

\begin{con}\label{conj01}
Suppose that $f=h+\overline{g}\in\mathcal{S}^0_\mathcal{H}(K)$. Then
$$\sup\limits_{f\in\mathcal{S}_\mathcal{H}^0(K)}\abs{a_2}=\frac{5K+3}{2K+2}.$$
The bound is expected to be sharp.
\end{con}

By Conjecture \ref{conj01}, Nowak’s conjecture and Proposition \ref{prop1}, it is natural to state the following conjecture
associated with the sharp order of the family $\mathcal{S}_\mathcal{H}(K)$ to belong to a harmonic Hardy space $h^p$.

\begin{con}\label{conj1}
Suppose that $f=h+\overline{g}\in\mathcal{S}_\mathcal{H}(K)$. Then
$$\sup\limits_{f\in\mathcal{S}_\mathcal{H}(K)}\abs{a_2}=\frac{3K+1}{K+1},$$
and, in particular, 
we have 
$$f\in h^p\quad \left(0<p<\frac{1}{2K}\right).$$
The bound for the order is expected to be sharp.
\end{con}

\begin{rem}\label{r0}
{\rm If one can prove Nowak’s
conjecture and the first part of Conjecture \ref{conj1}, it means that Problem \ref{p1} (Pavlovi\'{c}'s \textit{open problem}) can be \textit{solved} completely.}
\end{rem}

In what follows, we pose the conjecture associated with \textit{radius of covering theorem} of the family $\mathcal{S}^0_\mathcal{H}(K)$.
\begin{con} 
For $0<r<1$, $\D_r:=\{z: z\in\C\ {\it and}\ \abs z<r\}$ and  $K\geq1$, let $f\in {\mathcal{S}_\mathcal{H}^0(K)}$. Then 
$$R=\sup\limits_{0<r<1}\left\{r: \D_r\subseteq f(\D),\, f\in\mathcal{S}_\mathcal{H}^0(K)\right\}\geq\frac{K+1}{6K+2}.$$
The radius is expected to be sharp.
\end{con}

\begin{rem}\label{r3}
{\rm 
By a direct calculation, we see that $$\|S_{f_k}\|\leq\frac{19}{2}
\quad(0\leq k<1).$$}
\end{rem}







We observe that Graf \cite[Theorem 3]{g}, Hern\'{a}ndez and Mart\'{i}n \cite[Theorem~5]{hm} once showed that $\|S_f\|$ are bounded for locally univalent and univalent harmonic mappings, respectively,
but the questions about the sharp bounds are still \textit{open}. By Conjecture \ref{conj01} and Remark \ref{r3}, it is natural to pose the following conjecture:

\begin{con}\label{cc1}
Suppose that $f=h+\overline{g}\in\mathcal{S}^0_\mathcal{H}(K)$. Then $$\|S_f\|\leq\frac{19K^2+26K+3}{2(K+1)^2}.$$
The bound is expected to be sharp.
\end{con}

\vskip.20in
\section{Proof of Theorem \ref{t1}}

We start by noting the following variant of the classical Schwarz lemma.

\begin{lemma} {\rm (\cite{b})}\label{t111}
Suppose that $0<k\leq1$ and $\omega$ is analytic in $\D$ with $$\omega(0)=0\ \textit{and}\ \abs{\omega(z)}\leq k.$$ Then
 $$\abs{\omega(z)}\leq k\abs{z}\ \textit{and}\ \abs{\omega'(0)}\leq k.$$ Equality holds for the functions $$\omega(z)=k\,e^{i\theta}z\quad(\theta\in[0,2\pi)).$$
\end{lemma}

Now, we provide the sharp upper bound of $\abs{b_2}$ for the class $\mathcal{S}^0_\mathcal{H}(K)$, which will play a crucial role in determining the sharp upper bound of $\abs{a_2}$ of the family of harmonic $K$-quasiconformal mappings with bounded Schwarzian norm.

\begin{lemma}\label{t011}
Suppose that $f=h+\overline{g}\in\mathcal{S}^0_\mathcal{H}(K)$, where $$h(z)=z +\sum_{n=2}^{\infty}a_n z^n\ {\textit and} \ g(z)=\sum_{n=2}^{\infty}b_n z^n.$$ Then the sharp inequality \begin{equation}\label{22}\abs{b_2}\leq \frac{K-1}{2(K+1)}\end{equation}
holds.
\end{lemma}
\begin{proof}
For $f=h+\overline{g}\in\mathcal{S}^0_\mathcal{H}(K)$,
we see that its analytic dilatation $$\omega=\frac{g'}{h'}$$ satisfies the condition $$\abs{\omega(z)}\leq k\quad (0\leq k<1).$$
Moreover, if $f\in\mathcal{S}^0_\mathcal{H}(K)$, we know that $g'(0)=0$, i.e., $\omega(0)=0$. Thus, the assertion of
Lemma \ref{t111} shows that $$\abs{\omega'(0)}\leq k\quad (0\leq k<1).$$ By \cite[p. 87, Sect. 5.4]{d}, we get
 \begin{align*}
 \begin{split}\omega(z)&=\frac{g'(z)}{h'(z)}
 =\frac{2b_2z+3b_3z^2+\cdots}{1+2a_2z+\cdots}
=2b_2z+(3b_3-4a_2b_2)z^2+\cdots,
 \end{split}
\end{align*}
therefore, $$\abs{\omega'(0)}=\abs{2b_2}\leq k=\frac{K-1}{K+1},$$ which implies that the desired assertion of Lemma \ref{t011} is true. Equality holds for the functions $$\omega(z)=k\,e^{i\theta}z\quad\left(0\leq k<1;\, \theta\in[0,2\pi)\right).$$
This completes the proof of Lemma \ref{t011}.
\end{proof}

In what follows, we introduce the \textit{strongly affine transformation} for mappings \( f = h + \overline{g} \in \mathcal{S}_{\mathcal{H}}(K) \), defined by
\[
\widetilde{A}_\xi(f)(z) := \frac{f(z) - \overline{\xi f(z)}}{1 - \overline{\xi} g'(0)},
\]
where the parameter \( \xi \) satisfies
\begin{equation}
\label{111}
|\xi| \leq \eta \quad \left(0\leq\eta< 1\right).
\end{equation}

 A family $\mathcal{\widetilde{F}}$ is said to be a \textit{strongly affine and linear invariant family}
(SAL family) if it is closed under both the Koebe transform and strongly affine transformation.

\begin{lemma}\label{t012}
The families \( \mathcal{S}_{\mathcal H}(K) \) and \({\mathcal{S}_\mathcal{H}}(K,\lambda)\) belong to $\mathcal{\widetilde{F}}$ of strongly affine and linear invariant family.
\end{lemma}
\begin{proof}
It easily to check that the family \( \mathcal{S}_{\mathcal H}(K) \) (resp. \({\mathcal{S}_\mathcal{H}}(K,\lambda)\)) is a linear invariant family (see \cite[p. 9]{DR}).
Let $\widetilde{A}_{\xi}(f)(z)=H(z)+\overline{G}(z)$, and denote its complex dilatation by $$\widetilde{\omega}(z)=\frac{G'(z)}{H'(z)}.$$

Note that 
the constraint \[|\xi| \leq \frac{k - |\omega(z)|}{1 - k|\omega(z)|}<1 \quad \left(|\omega(z)| = \left|\frac{g'}{h'}\right| \leq k; \ 0\leq k < 1\right)\] on \(\xi\) is equivalent to
\[\left| \frac{\omega(z) - \xi}{1 - \overline{\xi} \omega(z)} \right| \leq k < 1.\] In conjunction with the definition of \(\widetilde{\omega}(z)\), we have \[|\widetilde{\omega}(z)| = \left| \frac{G'(z)}{H'(z)} \right| = \left| \frac{g' - \xi h'}{h' - \overline{\xi} g'} \right| = \left| \frac{\omega(z) - \xi}{1 - \overline{\xi} \omega(z)} \right| \leq k < 1,\]
which shows that the family \( \mathcal{S}_\mathcal{H}(K) \) (resp. \({\mathcal{S}_\mathcal{H}}(K,\lambda)\)) is a strongly affine invariant family. 
Thus, we deduce that the families \( \mathcal{S}_\mathcal{H}(K) \) and \({\mathcal{S}_\mathcal{H}}(K,\lambda)\) are {strongly affine and linear invariant families}. 
\end{proof}

We now turn our attention to the following result involving the order of the family \(\widehat{\mathcal{S}_\mathcal{H}}(K,\lambda)\).

\begin{prop} \label{c1}
Let $f\in\widehat{\mathcal{S}_\mathcal{H}}(K,\lambda)$. 
Then
\begin{align*}
\sup\limits_{f\in\widehat{\mathcal{S}_\mathcal{H}}(K,\lambda)}\abs{a_2}=\sqrt{1+\frac{\lambda}{2}+\frac{1}{2}\left(\frac{K-1}{K+1}\right)^2}+\frac{K-1}{2\left({K+1}\right)}. 
\end{align*}
\end{prop}

\begin{proof} 
By harmonic quasiconformal Koebe function $f_k$ and Remark \ref{r3}, 
we see that $f_k\in \widehat{\mathcal{S}_\mathcal{H}}(K,\lambda)$, and  
inequality \eqref{22} also holds for this family. Note that Theorem \ref{thmD} remains valid for the strongly family \(\widehat{\mathcal{S}_\mathcal{H}}(K,\lambda)\).
By Lemmas \ref{t011} and \ref{t012}, and using the similar method as in Theorem \ref{thmD},
 we get the assertion of Proposition \ref{c1}.
\end{proof}



In what follows, we are ready to give the proof of Theorem \ref{t1}. 
\begin{proof}[Proof of Theorem \ref{t1}]
If $f\in\widehat{\mathcal{S}_\mathcal{H}}(K,\lambda)$, 
it follows from Proposition \ref{c1} that
$$\sup\limits_{f\in\widehat{\mathcal{S}_\mathcal{H}}(K,\lambda)}\abs{a_2}=\sqrt{1+\frac{\lambda}{2}+\frac{1}{2}\left(\frac{K-1}{K+1}\right)^2}+\frac{K-1}{2\left({K+1}\right)}\geq1.$$
For convenience, we define a real function $\Phi(K)$ with one parameter $\lambda\, (\lambda\geq0)$ as follows: \begin{equation}\label{33}\Phi(K):=\sqrt{1+\frac{\lambda}{2}+\frac{1}{2}\left(\frac{K-1}{K+1}\right)^2}+\frac{K-1}{2\left({K+1}\right)}-2K\quad(K\geq1).\end{equation}

Now, we divide the proof into three cases.
\vskip.05in
(i) When $0\leq\lambda\leq6$ and $K\geq1$, we know that $\Phi(K)$ is a decreasing continuous function with respect to $K$
in $[1,+\infty)$,
and $$\Phi(1)=\sqrt{1+\frac{\lambda}{2}}-2\leq0,$$ which implies that $$\Phi(K)\leq0\quad(K\geq1).$$ Thus, we have
\begin{equation}\label{3}\varphi(K,\lambda)=\sqrt{1+\frac{\lambda}{2}+\frac{1}{2}\left(\frac{K-1}{K+1}\right)^2}+\frac{K-1}{2\left({K+1}\right)}\leq 2K,\end{equation}
that is, \begin{equation}\label{4}\frac{1}{\varphi(K,\lambda)}\geq\frac{1}{2K}.\end{equation}
By Proposition \ref{prop1},
it shows that $$f\in h^p\quad \left(0<p<\frac{1}{2K}\right).$$
\vskip.05in
(ii) When $\lambda>6$, we know that $$\Phi'(K)<0\quad (K\geq1).$$ Furthermore, we find that $$\Phi(1)=\sqrt{1+\frac{\lambda}{2}}-2>0$$
and $$\Phi\left(\lambda\right)<0.$$ Applying the zero point theorem to the real continuous function $\Phi(K)$ in bounded closed interval
$[1,\, \lambda]$, we deduce that the equation \begin{equation*}\sqrt{1+\frac{\lambda}{2}+\frac{1}{2}\left(\frac{K-1}{K+1}\right)^2}+\frac{K-1}{2\left({K+1}\right)}=2K\quad(K\geq1;\, \lambda>6),\end{equation*} 
which is equivalent to the equation \eqref{2}, exists a unique solution $K_1$ in $(1,+\infty)$.
If $K\geq K_1>1$, we know that
 \eqref{4} also holds. By Proposition \ref{prop1}, it implies that $$f\in h^p\quad \left(0<p<\frac{1}{2K}\right).$$
\vskip.05in
(iii) When $\lambda>6$ and $1\leq K<K_1$, we see that
\begin{equation*}\label{5}\varphi(K,\lambda)=\sqrt{1+\frac{\lambda}{2}+\frac{1}{2}\left(\frac{K-1}{K+1}\right)^2}+\frac{K-1}{2\left({K+1}\right)}>2K,\end{equation*}
or equivalently, $$\frac{1}{\varphi(K,\lambda)}<\frac{1}{2K}.$$ By Proposition \ref{prop1},
it means that $$f\in h^p\quad \left(0<p<\frac{1}{\varphi(K,\lambda)}\right).$$
The proof of Theorem \ref{t1} is thus completed.
\end{proof}

\vskip .10in
	\noindent{\bf  Acknowledgements.}
The authors 
thank 
Daoud Bshouty, 
David Kalaj, 
Gang Liu,
Zhi-Hong Liu, Saminathan Ponnusamy, 
Toshiyuki Sugawa,  
Yu-Dong Wu, 
and 
 Jian-Feng Zhu
for their helpful conversations and useful comments in different stages of the preparation of this paper.

\vskip .10in
	\noindent{\bf  Funding.}
Z.-G. Wang was partially supported by the \textit{Key Project of Education Department of Hunan Province} under Grant no. 25A0668, and
the \textit{Natural Science Foundation of Changsha} under Grant no. kq2502003
of the P. R. China. X.-Y. Wang was partially supported by the \textit{National Scholarship Council of China} under Grant
no. 202306840137 of the P. R. China.
A. Rasila was partially supported by \textit{Natural Science Foundation of Guangdong Province} under Grant no. 2024A1515010467 of the P. R. China, and \textit{Li Ka Shing Foundation} under Grant no. 2024LKSFG06. 
J.-L. Qiu was partially supported by the \textit{Graduate Research Innovation Project of Hunan Province} under Grant no. CX20240803 of the P. R. China.


\vskip .10in
\noindent{\bf Conflicts of interest.} The authors declare that they have no conflict of interest.

\vskip .10in
\noindent{\bf Data availability statement.}  Data sharing is not applicable to this article as no datasets were generated or analysed during the current study.


\vskip.05in
\begin{thebibliography}{99}
\bibitem{aap}
{\small Y. Abu-Muhanna, R. M. Ali, and S. Ponnusamy, The spherical metric and univalent harmonic mappings,
\textit{Monatsh. Math.} \textbf{188} (2019), 703--716.}








\vskip.05in
\bibitem{al}
{\small Y. Abu-Muhanna and A. Lyzzaik, The boundary behavior of harmonic univalent maps, \textit{Pacific J. Math.} \textbf{141} (1990), 1--20.}


\vskip.05in
\bibitem{am}
{\small A. Aleman and M. J. Mart\'{\i}n, Convex harmonic mappings are not necessarily in $h^{1/2}$, \textit{Proc. Amer. Math. Soc.} \textbf{143} (2015), 755--763.}









\vskip.05in
\bibitem{ak}
{\small K. Astala and P. Koskela, $H^p$-theory for quasiconformal mappings, \textit{Pure Appl. Math. Q.} \textbf{7} (2011), 19--50.}








\vskip.05in
\bibitem{b}
{\small M. Borovikov, On Koebe radius and coefficients estimate for univalent harmonic mappings, 
\textit{J. Anal.} \textbf{33} (2025), 2275--2284.}










\vskip.05in
\bibitem{cp}
{\small S. Chen and S. Ponnusamy, Koebe type theorems and pre-Schwarzian of harmonic $K$-quasiconformal mappings, and their applications, \textit{Acta Math. Sin. (Engl. Ser.)} \textbf{38} (2022), 1965--1980.}


\vskip.05in
\bibitem{cdo}
{\small M. Chuaqui, P. Duren, and B. Osgood, The Schwarzian derivative for harmonic mappings, \textit{J. Anal. Math.} \textbf{91} (2003), 329--351.}


\vskip.05in
\bibitem{chm}
{\small M. Chuaqui, R. Hern\'{a}ndez, and M. J. Mart\'{i}n, Affine and linear invariant families of harmonic mappings, \textit{Math. Ann.} \textbf{367} (2017), 1099--1122.}

\vskip.05in
\bibitem{cs}
{\small J. G. Clunie and T. Sheil-Small, Harmonic univalent functions, \textit{Ann. Acad. Sci. Fenn. Ser. A. I Math.} \textbf{9} (1984), 3--25.}


\vskip.05in
\bibitem{ds}
{\small S. Das and A. Sairam Kaliraj, Integral mean estimates for univalent and locally univalent harmonic mappings,
\textit{Canad. Math. Bull.} \textbf{67} (2024), 655--669.}


\vskip.05in
\bibitem{de}
{\small L. de Branges, A proof of the Bieberbach conjecture, \textit{Acta Math.} \textbf{154} (1985), 137--152.}


\vskip.05in
\bibitem{dhr}
{\small S. Das, J. Huang, and A. Rasila, Hardy spaces of harmonic quasiconformal mappings and Baernstein's theorem, \textit{Bull. Sci. Math.} \textbf{208} (2026), 103789.}

\vskip.05in
\bibitem{DR} 
{\small S. Das, A.Rasila,
On harmonic quasiregular mappings in Bergman spaces,
\textit{Potential Anal.} \textbf{64} (2026), 26.}


\vskip.05in
\bibitem{du}
{\small P. L. Duren, \textit{Univalent functions}, Grundlehren der Mathematischen Wissenschaften, Vol. 259.
Springer, New York, 1983.}

\vskip.05in
\bibitem{d}
{\small P. Duren, \textit{Harmonic mappings in the plane}, Cambridge University
Press, Cambridge, 2004.}






\vskip.05in
\bibitem{g}
{\small S. Y. Graf, On the Schwarzian norm of harmonic mappings, \textit{Probl. Anal. Issues Anal.} \textbf{5(23)} (2016), 20--32.}








\vskip.05in
\bibitem{hm}
{\small R. Hern\'{a}ndez and M. J. Mart\'{i}n, Pre-Schwarzian and Schwarzian derivatives of harmonic mappings,
\textit{J. Geom. Anal.} \textbf{25} (2015), 64--91.}











\vskip.05in
\bibitem{ka}
{\small D. Kalaj, Muckenhoupt weights and Lindel\"{o}f theorem for harmonic mappings, \textit{Adv. Math.}
\textbf{280} (2015), 301--321.}

\vskip.05in
\bibitem{ka2}
{\small D. Kalaj, On Riesz type inequalities for harmonic mappings on the unit disk, \textit{Trans. Amer. Math. Soc.} \textbf{372} (2019), 4031--4051.}



 










\vskip.05in
\bibitem{ln}
{\small
J. Laitila, P. J. Nieminen, E. Saksman, and  H.-O. Tylli, Rigidity of composition operators on the Hardy space $H^p$, \textit{Adv. Math.} \textbf{319} (2017), 610--629.}

\vskip.05in
\bibitem{lp}
{\small
P. Li and S. Ponnusamy, On the coefficients estimate of $K$-quasiconformal harmonic mappings, 
\textit{Res. Math. Sci.} \textbf{13} (2026), Paper No. 16. }


\vskip.05in
\bibitem{lz}
{\small J. Liu and J.-F. Zhu, Riesz conjugate functions theorem for harmonic quasiconformal mappings, \textit{Adv. Math.} \textbf{434} (2023), Paper No. 109321, 27 pp.}









\vskip.05in
\bibitem{mb}
{\small P. Melentijevi\'{c} and V. Bo\v{z}in,  Sharp Riesz-Fej\'{e}r inequality for harmonic Hardy spaces, \textit{Potential Anal.} \textbf{54} (2021), 575--580.}

\vskip.05in
\bibitem{n}
{\small M. Nowak, Integral means of univalent harmonic maps, \textit{Ann. Univ. Mariae Curie-Skłodowska Sect. A} \textbf{50} (1996), 155--162.}

\vskip.05in
\bibitem{psz}
{\small D. Partyka, K. Sakan, and J.-F. Zhu, Quasiconformal harmonic mappings with the convex holomorphic part, \textit{Ann. Acad. Sci. Fenn. Math.}
\textbf{43} (2018), 401--418.}

\vskip.05in
\bibitem{pav}
{\small M. Pavlovi\'{c}, \textit{Function classes on the unit disc: an introduction}, De Gruyter Studies in Mathematics, Vol. 52, De Gruyter, Berlin, 2014.}




\vskip.05in
\bibitem{p}
{\small C. Pommerenke, Linear-invariante Familien analytischer Funktionen I, \textit{Math. Ann.} \textbf{155} (1964), 108--154.}


\vskip.05in
\bibitem{pqw}
{\small S. Ponnusamy, J. Qiao, and X. Wang, Uniformly locally univalent harmonic mappings, \textit{Proc. Indian Acad. Sci. Math. Sci.} \textbf{128} (2018), Paper No. 32, 14 pp.}







\vskip.05in
\bibitem{s}
{\small T. Sheil-Small, Constants for planar harmonic mappings, \textit{J. London Math.
Soc.} \textbf{42} (1990), 237--248.}






\vskip.05in
\bibitem{srj}
{\small Y. Sun, A. Rasila, and Y.-P. Jiang, Linear combinations of harmonic quasiconformal mappings convex
in one direction, \textit{Kodai Math. J.}  \textbf{39} (2016), 366--377.}


\vskip.05in
\bibitem{t}
{\small V. Todor\v{c}evi\'{c}, \textit{Harmonic quasiconformal mappings and hyperbolic type metrics}, Springer, Cham, 2019.}






\vskip.05in
\bibitem{wqr}
{\small Z.-G. Wang, J.-L. Qiu, and A. Rasila, On Koebe-type functions for harmonic quasiconformal mappings, to appear in \textit{Anal. Math.} 2026.}

\vskip.05in
\bibitem{wsj}
{\small Z.-G. Wang, L. Shi, and Y.-P. Jiang, On harmonic $K$-quasiconformal mappings associated with asymmetric vertical strips,
\textit{Acta Math. Sin. (Engl. Ser.)} \textbf{31} (2015), 1970--1976.}









\end{thebibliography}
\end{document}